\documentclass[preprint,12pt]{elsarticle}

\usepackage{natbib}

\usepackage{graphicx}%
\usepackage{multirow}%
\usepackage{amsmath,amssymb,amsfonts}%
\usepackage{amsthm}%
\usepackage{mathrsfs}%
\usepackage[title]{appendix}%
\usepackage{xcolor}%
\usepackage{textcomp}%
\usepackage{manyfoot}%
\usepackage{booktabs}%
\usepackage{algorithm}%
\usepackage{algorithmicx}%
\usepackage{algpseudocode}%
\usepackage{listings}%
\usepackage{caption}
\usepackage[utf8]{inputenc}
\usepackage{comment}

\journal{Journal of Combinatorial Optimization}

\begin{document}

\newcommand{\Z}{{\mathbb Z}}
\newcommand{\N}{{\mathbb N}}
\newcommand{\Hi}{{\mathbb H}}
\newcommand{\R}{{\mathbb R}}
\newcommand{\Q}{{\mathbb Q}}
\newcommand{\ICG}{\mathrm{ICG}}
\newcommand{\WCG}{\mathrm{WCG}}
\newcommand{\WICG}{\mathrm{WICG}}

\newtheorem{theorem}{\bf Theorem}[section]
\newtheorem{corollary}[theorem]{\bf Corollary}
\newtheorem{lemma}[theorem]{\bf Lemma}
\newtheorem{proposition}[theorem]{\bf Proposition}
\newtheorem{conjecture}[theorem]{\bf Conjecture}
\newtheorem{remark}[theorem]{\bf Remark}
\newtheorem{problem}[theorem]{\bf Problem}
\newtheorem{definition}[theorem]{\bf Definition}

\newcommand{\QED} {\hfill$\square$}

\def\slika #1{\begin{center} \epsffile{#1} \end{center}}

\begin{frontmatter}

\title{On the maximum $\sigma$-irregularity of trees with given order and maximum degree}

\author[address1]{Milan Ba\v si\'c\corref{mycorrespondingauthor}}
\cortext[mycorrespondingauthor]{Corresponding author}
\ead{basic\_milan@yahoo.com}
\address[address1]{University of Ni\v s, Ni\v s, Serbia}

\begin{abstract}
The $\sigma$-irregularity index of a graph is defined as the sum of squared degree
differences over all edges and provides a sensitive measure of structural
heterogeneity.
In this paper, we study the problem of maximizing $\sigma(T)$ among all trees of
fixed order $n$ and prescribed maximum degree $\Delta\ge4$.
By expressing the problem in terms of edge--degree multiplicities, we derive a
linear programming formulation and analyze its dual.
This approach yields sharp upper bounds for $\sigma(T)$ and leads to a detailed
description of extremal degree--pair distributions.
We show that the extremal problem can be completely resolved for the congruence
classes $n\equiv1\pmod{\Delta}$ and $n\equiv0\pmod{\Delta}$.
When $n\equiv1\pmod{\Delta}$, the linear program admits an integral optimal
solution, and the bound for $\sigma(T)$ is tight.
When $n\equiv0\pmod{\Delta}$, the linear relaxation is not attainable by any
tree; nevertheless, by introducing a penalty function derived from dual slack
variables, we determine the exact maximum value of $\sigma(T)$.
In both cases, all extremal trees are characterized explicitly and consist
exclusively of vertices of degrees $1$, $2$, and $\Delta$, with edges incident to
$\Delta$-vertices playing a dominant role.
\end{abstract}


\end{frontmatter}



\section{Introduction}

A graph is called \emph{irregular} if not all of its vertices have the same degree.
Quantifying the extent to which a graph deviates from regularity has long been of
interest in both theoretical and applied graph theory.
One of the earliest quantitative measures of degree heterogeneity is the
\emph{degree variance} proposed by Snijders~\cite{Snijders1981}.
Subsequently, Albertson~\cite{Albertson1997} introduced the now classical
irregularity index
\[
\mathrm{irr}(G)=\sum_{uv\in E(G)} |d_G(u)-d_G(v)|,
\]
which measures the total imbalance of degrees across adjacent vertices.
This invariant, commonly referred to as the \emph{Albertson irregularity},
has been widely studied; see, for instance,
\cite{Albertson1997,Abdo2014,Hansen2005} and the references therein.

Degree–based irregularity measures have proved particularly useful in applications.
In network science, they serve as indicators of structural heterogeneity in
communication and social networks~\cite{Estrada2010}.
In mathematical chemistry, such indices are employed as molecular descriptors
in quantitative structure–property and structure–activity relationships
(QSPR/QSAR), where they capture variations in local bonding environments
\cite{Gutman2005,Reti2018QSPR}.
These applications have motivated the introduction of alternative irregularity
indices that are more sensitive to large local degree differences.

Among these variants, a prominent role is played by the
\emph{$\sigma$-irregularity index}, or simply the \emph{$\sigma$-index}, defined
for a graph $G=(V,E)$ by
\[
\sigma(G)=\sum_{uv\in E(G)}\bigl(d_G(u)-d_G(v)\bigr)^2.
\]
This index may be viewed as a quadratic analogue of Albertson’s irregularity.
By squaring degree differences, $\sigma(G)$ places greater weight on edges
joining vertices with highly disparate degrees, making it a particularly
sensitive measure of structural imbalance.
The $\sigma$-index was systematically investigated in
\cite{Abdo2018}, where extremal graphs maximizing $\sigma(G)$ among all graphs of
fixed order were characterized.
Further fundamental properties, including the solution of the inverse
$\sigma$-index problem, were obtained in \cite{Gutman2018}.
Additional comparisons with other irregularity measures and applications appear
in \cite{Reti2019}, while Arif et al.~\cite{Arif2023} studied $\sigma(G)$ and
related indices for graph families characterized by two main eigenvalues and
applied the resulting values in QSPR analysis.

A natural and active direction of research concerns extremal problems for
$\sigma(G)$ under additional structural constraints.
Once the extremal graphs of fixed order are known, it is natural to ask how the
problem behaves when the admissible class of graphs is restricted.
Among the most important such restrictions is the class of trees, which is both
mathematically fundamental and highly relevant in applications.
In chemical graph theory, trees model acyclic molecular structures; in this
context, one often speaks of \emph{chemical trees} when vertex degrees are
bounded by~$4$.
The $\sigma$-irregularity of chemical trees was studied in
\cite{Kovijanic2024}, where a complete characterization of trees with maximum
$\sigma$ among all chemical trees of fixed order was obtained.

More recently, Dimitrov \emph{et al.}~\cite{Dimitrov2026} initiated a systematic
study of $\sigma$-extremal trees under a prescribed maximum degree $\Delta=5$.
They established general structural properties of $\sigma$-maximal trees for
arbitrary $\Delta\ge3$ and carried out a detailed analysis for the case
$\Delta=5$.
In that setting, they showed that every $n$-vertex tree maximizing $\sigma(G)$
with $\Delta(G)=5$ contains only vertices of degrees $1$, $2$, and~$5$.
Moreover, almost all edges in such extremal trees join a leaf to a degree-$5$
vertex or a degree-$5$ vertex to a degree-$2$ vertex, while edges between two
degree-$5$ vertices or two degree-$2$ vertices occur only exceptionally.
These findings support a general heuristic: for fixed maximum degree $\Delta$,
trees with large $\sigma$-irregularity tend to exhibit strongly polarized degree
distributions, with many vertices of degree $\Delta$ and many leaves, while
vertices of degree~$2$ appear in comparable number and serve primarily as
connectors between the two extremes.

In this paper, we study the problem of maximizing the $\sigma$-irregularity index
within the class of trees of fixed order $n$ and prescribed maximum degree
$\Delta\ge4$.
The problem appears difficult in full generality, as the extremal structure
depends sensitively on arithmetic relations between $n$ and $\Delta$.
To address this, we reformulate the extremal problem as a linear optimization
problem over edge–degree multiplicities and analyze its dual.
This approach yields sharp upper bounds for $\sigma(T)$ and allows one to derive
the extremal trees explicitly in terms of their degree distributions and
edge--degree multiplicities.

Our main results show that the extremal problem can be completely resolved for
the congruence classes $n\equiv1\pmod{\Delta}$ and $n\equiv0\pmod{\Delta}$.
When $n\equiv1\pmod{\Delta}$, the linear program admits an integral optimal
solution, and the resulting bound for $\sigma(T)$ is tight and attained by a
unique extremal degree–pair configuration.
When $n\equiv0\pmod{\Delta}$, the linear relaxation is not attainable by any tree.
Nevertheless, by introducing an appropriate penalty function derived from the
dual slack variables, we determine the exact maximum value of $\sigma(T)$ and
obtain a complete structural characterization of all extremal trees in this
case. In both congruence classes treated here, the extremal trees consist exclusively
of vertices of degrees $1$, $2$, and $\Delta$, with edges of types $(1,\Delta)$
and $(2,\Delta)$ playing a dominant role.
Preliminary evidence, supported by exhaustive computer search for larger values
of $\Delta$, suggests that for increasing $\Delta$ new extremal configurations
may appear that involve additional intermediate degrees, reflecting a gradual
increase of the minimum penalty value as the residue of $n$ modulo $\Delta$
moves away from $0$ and $1$.

The paper is organized as follows.
In Section~2 we introduce notation and derive a linear programming formulation
for the maximization of $\sigma(T)$ over trees with given order and maximum
degree.
In Section~3 we solve the corresponding linear program and identify the optimal degree and
edge degree–pair distributions, discussing the role of integrality and congruence
conditions on~$n$.
In Section~4 we introduce a penalty function and analyze the case
$n\equiv0\pmod{\Delta}$, establishing sharp bounds and uniqueness results.
Finally, in Subsection~4.1 we provide explicit constructions of all extremal trees for
both congruence classes $n\equiv1\pmod{\Delta}$ and $n\equiv0\pmod{\Delta}$,
thereby complementing the abstract optimization results with concrete
combinatorial realizations.

\section{Linear programming formulation of the problem}

Let $T$ be a tree on $n$ vertices with maximum degree $\Delta$.  
For $1\le i\le \Delta$, denote by $n_i$ the number of vertices of degree $i$, and for
$1\le i\le j\le \Delta$, denote by $m_{i,j}$ the number of edges joining a vertex
of degree $i$ to a vertex of degree $j$.  
Thus,
\[
\sum_{1\le i\le j\le \Delta} m_{i,j} = |E(T)| = n-1,
\]
and the quantities $m_{i,j}$ capture the complete degree–pair structure of $T$.

\begin{definition}
\label{def:sigma}
For a connected graph $G$ with edge set $E(G)$, the \emph{sigma-regularity index}
is defined by
\begin{equation}
\label{eq:def_sigma}
\sigma(G)\;=\;\sum_{uv\in E(G)}\bigl(d_G(u)-d_G(v)\bigr)^2.
\end{equation}
\end{definition}

The expression for $\sigma(G)$ can be restated as follows. For $1\le i\le j\le \Delta$,
let $m_{i,j}$ denote the number of edges joining a vertex of degree $i$ to a vertex
of degree $j$. Then
\[
\sigma(G)
=\sum_{1\le i\le j\le \Delta} m_{i,j}\,(i-j)^2.
\]

\medskip
The variables $n_i$ and $m_{i,j}$ are related by the following system of identities, which hold for every tree $T$:

\begin{align}
n_1 + n_2 + \cdots + n_{\Delta} &= n, \label{eq:tree-sum-ni}\\
n_1 + 2n_2 + \cdots + \Delta n_{\Delta} &= 2n-2, \label{eq:tree-sum-deg}\\
m_{12} + m_{13} + \cdots + m_{1\Delta} &= n_1, \label{eq:tree-ni-1}\\
m_{12} + 2m_{22} + m_{23} + \cdots + m_{2\Delta} &= 2n_2, \label{eq:tree-ni-2}\\
m_{13} + m_{23} + 2m_{33} + \cdots + m_{3\Delta} &= 3n_3, \label{eq:tree-ni-3}\\
\intertext{\centering$\vdots$}
m_{1i}+m_{2i}+\cdots+m_{i-1,i}+2m_{ii}+m_{i,i+1}+\cdots+m_{i\Delta}
&= i n_i, \label{eq:tree-ni-i}\\
\intertext{\centering$\vdots$}
m_{1\Delta} + m_{2\Delta} + m_{3\Delta} + \cdots + 2m_{\Delta\Delta}
&= \Delta n_\Delta. \label{eq:tree-ni-Delta}
\end{align}

Equations \eqref{eq:tree-ni-1}--\eqref{eq:tree-ni-Delta} express each $n_i$ in terms of the edge–degree incidence numbers $m_{i,j}$.
Substituting these expressions into \eqref{eq:tree-sum-ni} and \eqref{eq:tree-sum-deg} yields two linear constraints involving only the variables $m_{i,j}$.  
A straightforward computation shows that \eqref{eq:tree-sum-deg} becomes
\begin{equation}
\sum_{1\le i\le j\le \Delta} m_{i,j} = n-1,
\label{eq:LP-constraint-sum}
\end{equation}
reflecting the fact that every tree has exactly $n-1$ edges, while 
\eqref{eq:tree-sum-ni} transforms into
\begin{equation}
\sum_{1\le i\le j\le \Delta}
\Bigl(\frac{1}{i}+\frac{1}{j}\Bigr)\, m_{i,j} = n.
\label{eq:LP-constraint-w}
\end{equation}
These two constraints encapsulate all global restrictions that the tree structure imposes on the family $\{m_{i,j}\}$.

\medskip

Consequently, maximizing $\sigma(T)$ over all trees of order $n$ and maximum degree $\Delta$ becomes equivalent to maximizing the linear functional
\[
\sum_{1\le i\le j\le \Delta} (i-j)^2 m_{i,j}
\]
over all nonnegative solutions $\{m_{i,j}\}_{1\le i\le j\le \Delta}$ satisfying 
\eqref{eq:LP-constraint-sum} and \eqref{eq:LP-constraint-w}.
This is precisely the linear program studied in Theorem~\ref{thm:SO-LP-1-Delta}, whose solution characterizes the extremal degree–pair distributions of trees maximizing~$\sigma$.
Thus, maximizing $\sigma(T)$ over all trees of order $n$ and maximum degree
$\Delta$ is equivalent to solving the linear program
\eqref{eq:primal-objective}--\eqref{eq:primal-nonneg}.



\section{Solution of the linear program}

We begin by determining the optimal solution of the linear relaxation,
without imposing integrality conditions on the variables $m_{i,j}$.

\begin{remark}
We shall use the following standard facts from linear programming.
If a primal linear program and its dual admit feasible solutions whose
objective values coincide, then both solutions are optimal (strong duality).
Moreover, in an optimal primal--dual pair, a primal variable can be positive
only if the corresponding dual constraint is tight (complementary slackness).
\end{remark}

\begin{theorem}
\label{thm:SO-LP-1-Delta}
Let $\Delta\ge 4$ and $n\ge \Delta+2$ be integers. Consider the linear program
\begin{align}
\max\ & \sigma(T) \;=\; \sum_{1\le i\le j\le \Delta} (i-j)^2\, m_{i,j},
\label{eq:primal-objective}\\[1ex]
\text{subject to }&
\sum_{1\le i\le j\le \Delta} \Bigl(\frac{1}{i}+\frac{1}{j}\Bigr) m_{i,j} \;=\; n,
\label{eq:primal-constraint-w}\\
&
\sum_{1\le i\le j\le \Delta} m_{i,j} \;=\; n-1,
\label{eq:primal-constraint-sum}\\
&
m_{i,j}\;\ge\; 0
\quad\text{for all }1\le i\le j\le \Delta.
\label{eq:primal-nonneg}
\end{align}
Then the problem admits a unique optimal solution in which
\[
m_{1,\Delta} \;=\; \frac{(\Delta-2)n + (\Delta+2)}{\Delta}, 
\qquad
m_{2,\Delta} \;=\; \frac{2(n-\Delta-1)}{\Delta},
\]
and
\[
m_{i,j}=0 \quad\text{for all } (i,j)\notin\{(1,\Delta),(2,\Delta)\}.
\]
\end{theorem}

\begin{proof}
Define
\[
w_{i,j} := \frac{1}{i}+\frac{1}{j},
\qquad
\sigma_{i,j} := (i-j)^2.
\]
Then the primal problem \eqref{eq:primal-objective}--\eqref{eq:primal-nonneg}
can be written in the standard form
\[
\max \sum_{1\le i\le j\le \Delta} \sigma_{i,j} m_{i,j}
\quad\text{subject to}\quad
\sum_{i\le j} w_{i,j} m_{i,j} = n,\;
\sum_{i\le j} m_{i,j} = n-1,\;
m_{i,j}\ge 0.
\]

\medskip\noindent
\emph{Dual problem.}
The dual linear program associated with this primal problem is
\begin{equation}
\label{eq:dual-problem}
\begin{aligned}
\min\ & D(\lambda,\mu) \;=\; n\lambda + (n-1)\mu,\\
\text{subject to }&
\lambda w_{i,j} + \mu \;\ge\; \sigma_{i,j},
\quad 1\le i\le j\le \Delta,
\end{aligned}
\end{equation}
where $\lambda,\mu\in\mathbb{R}$ are the dual variables corresponding to
constraints \eqref{eq:primal-constraint-w} and \eqref{eq:primal-constraint-sum},
respectively.

For convenience, we define the dual slack function
\begin{equation}
\label{eq:Fij-def}
F(i,j)
\;:=\;
\lambda w_{i,j} + \mu - \sigma_{i,j}.
\end{equation}
Dual feasibility is equivalent to $F(i,j)\ge 0$ for all $1\le i\le j\le \Delta$.

\medskip\noindent
\emph{Choice of dual variables.}
We now determine $\lambda$ and $\mu$ by requiring that the dual constraints
be tight for the index pairs $(1,\Delta)$ and $(2,\Delta)$, that is,
\begin{equation}
\label{eq:dual-equalities-1-2}
F(1,\Delta)=0,
\qquad
F(2,\Delta)=0.
\end{equation}
Equivalently,
\[
\lambda w_{1,\Delta} + \mu = \sigma_{1,\Delta} = (\Delta-1)^2,
\qquad
\lambda w_{2,\Delta} + \mu = \sigma_{2,\Delta} = (\Delta-2)^2.
\]

Since
\[
w_{1,\Delta} = 1+\frac{1}{\Delta},
\qquad
w_{2,\Delta} = \frac12+\frac{1}{\Delta},
\]
subtracting the two equalities gives
\[
(\Delta-1)^2 - (\Delta-2)^2
=
\lambda\Bigl(1+\frac{1}{\Delta}-\frac12-\frac{1}{\Delta}\Bigr)
=
\frac{\lambda}{2},
\]
hence
\begin{equation}
\label{eq:lambda}
\lambda
=
2\bigl((\Delta-1)^2-(\Delta-2)^2\bigr)
=
4\Delta-6.
\end{equation}
Substituting this value into the first equality yields
\begin{equation}
\label{eq:mu}
\mu
=
(\Delta-1)^2 - (4\Delta-6)\Bigl(1+\frac{1}{\Delta}\Bigr)
=
\Delta^2 - 6\Delta + 3 + \frac{6}{\Delta}.
\end{equation}

With this choice of $\lambda$ and $\mu$, the slack function
\eqref{eq:Fij-def} takes the explicit form
\begin{equation}
\label{eq:Fij-explicit}
F(i,j)
=
(4\Delta-6)\Bigl(\frac{1}{i}+\frac{1}{j}\Bigr)
+ \Delta^2 - 6\Delta + 3 + \frac{6}{\Delta}
- (i-j)^2.
\end{equation}

\medskip\noindent
\emph{Verification of dual feasibility.}

\smallskip\noindent
\emph{Case 1: $i=1$, $1\le j\le \Delta$.}
From \eqref{eq:Fij-explicit},
\[
F(1,j)
=
(4\Delta-6)\Bigl(1+\frac{1}{j}\Bigr)
+ \Delta^2 - 6\Delta + 3 + \frac{6}{\Delta}
- (1-j)^2.
\]
A direct computation yields
\begin{equation}
\label{eq:F1j-factor}
\Delta j\,F(1,j)
=
(\Delta-j)
\left(
\Delta j^2 + \Delta^2 j - 2\Delta j + 4\Delta - 6
\right).
\end{equation}
For $1\le j\le \Delta-1$ we have $\Delta-j>0$, so it suffices to analyze
\[
G_1(j)
=
\Delta j^2 + (\Delta^2-2\Delta)j + 4\Delta - 6.
\]
This quadratic polynomial is strictly convex, with vertex at
\[
j_0=-\frac{\Delta^2-2\Delta}{2\Delta}=-\frac{\Delta-2}{2}\le 0,
\]
and hence is strictly increasing on $[0,\infty)$. Therefore,
\[
\min_{1\le j\le \Delta-1} G_1(j)=G_1(1)=\Delta^2+3\Delta-6>0
\quad(\Delta\ge 4).
\]
It follows from \eqref{eq:F1j-factor} that
\[
F(1,j)>0 \quad\text{for }1\le j\le \Delta-1,
\qquad
F(1,\Delta)=0.
\]

\smallskip\noindent
\emph{Case 2: $i=2$, $2\le j\le \Delta$.}
Similarly,
\[
F(2,j)
=
(4\Delta-6)\Bigl(\frac12+\frac{1}{j}\Bigr)
+ \Delta^2 - 6\Delta + 3 + \frac{6}{\Delta}
- (2-j)^2,
\]
and a direct computation yields
\begin{equation}
\label{eq:F2j-factor}
\Delta j\,F(2,j)
=
(\Delta-j)
\left(
\Delta j^2 + \Delta^2 j - 4\Delta j + 4\Delta - 6
\right).
\end{equation}
For $2\le j\le \Delta-1$, consider
\[
G_2(j)
=
\Delta j^2 + (\Delta^2-4\Delta)j + 4\Delta - 6.
\]
This is again strictly convex with vertex
\[
j_0=-\frac{\Delta^2-4\Delta}{2\Delta}=-\frac{\Delta-4}{2}\le 0,
\]
so $G_2$ is increasing on $[0,\infty)$ and
\[
\min_{2\le j\le \Delta-1} G_2(j)=G_2(2)=2\Delta^2-6>0.
\]
Hence
\[
F(2,j)>0 \quad\text{for }2\le j\le \Delta-1,
\qquad
F(2,\Delta)=0.
\]

\smallskip\noindent
\emph{Case 3: $3\le i\le j\le \Delta$.}
Using $1/i\ge 1/\Delta$ and $1/j\ge 1/\Delta$, we obtain
\[
F(i,j)
\ge
(4\Delta-6)\frac{2}{\Delta}
+ \Delta^2 - 6\Delta + 3 + \frac{6}{\Delta}
- (i-j)^2.
\]
Since $0\le j-i\le \Delta-3$, we have $(i-j)^2\le (\Delta-3)^2$, and therefore
\begin{align*}
F(i,j)
&\ge
\Delta^2 - 6\Delta + 3 + \frac{6}{\Delta}
+ \frac{8\Delta-12}{\Delta}
- (\Delta-3)^2\\
&=
2-\frac{6}{\Delta}>0
\qquad(\Delta\ge 4).
\end{align*}

Combining the three cases, we conclude that
\[
F(i,j)\ge 0 \quad\text{for all }1\le i\le j\le \Delta,
\]
with equality if and only if $(i,j)\in\{(1,\Delta),(2,\Delta)\}$.
Thus $(\lambda,\mu)$ is dual feasible.

\medskip\noindent
\emph{Construction of a primal feasible solution.}
Assume $m_{i,j}=0$ for all $(i,j)\notin\{(1,\Delta),(2,\Delta)\}$.
Then constraints \eqref{eq:primal-constraint-w}--\eqref{eq:primal-constraint-sum}
reduce to
\begin{equation}
\label{eq: reduced system}
m_{1,\Delta}+m_{2,\Delta}=n-1,
\qquad
w_{1,\Delta}m_{1,\Delta}+w_{2,\Delta}m_{2,\Delta}=n.
\end{equation}
Solving this system yields
\[
m_{1,\Delta}=\frac{(\Delta-2)n+(\Delta+2)}{\Delta},
\qquad
m_{2,\Delta}=\frac{2(n-\Delta-1)}{\Delta}.
\]
Since $n\ge\Delta+2$, both values are positive.

\medskip\noindent
\emph{Optimality and uniqueness.}
For the primal solution constructed above, we compute
\begin{align*}
\sigma(T)
&=
\sum_{1\le i\le j\le \Delta} \sigma_{i,j} m_{i,j}\\
&=
\sigma_{1,\Delta}m_{1,\Delta}+\sigma_{2,\Delta}m_{2,\Delta}
\qquad(\text{all other }m_{i,j}=0)\\
&=
(\lambda w_{1,\Delta}+\mu)m_{1,\Delta}+(\lambda w_{2,\Delta}+\mu)m_{2,\Delta}
\qquad(\text{by }F(1,\Delta)=F(2,\Delta)=0)\\
&=
\sum_{1\le i\le j\le \Delta} (\lambda w_{i,j}+\mu)\, m_{i,j}\\
&=
\lambda\sum_{i\le j} w_{i,j} m_{i,j}
+
\mu\sum_{i\le j} m_{i,j}
=
\lambda n + \mu (n-1)
=
D(\lambda,\mu),
\end{align*}
which coincides with the dual objective value at the feasible dual solution
$(\lambda,\mu)$ given in \eqref{eq:lambda} and \eqref{eq:mu}.
By strong duality, both the primal and dual solutions are optimal.

Furthermore, we have shown that
\[
F(i,j)=\lambda w_{i,j}+\mu-\sigma_{i,j}>0
\quad\text{for all }(i,j)\notin\{(1,\Delta),(2,\Delta)\},
\]
while $F(1,\Delta)=F(2,\Delta)=0$.
By complementary slackness, this implies that in any optimal primal solution
one must have
\[
m_{i,j}=0
\quad\text{for all }(i,j)\notin\{(1,\Delta),(2,\Delta)\}.
\]
Since the reduced system \eqref{eq: reduced system}
has a unique solution, the optimal primal solution is unique.
This completes the proof.

\end{proof}

It is worth noting that the optimal solution described in
Theorem~\ref{thm:SO-LP-1-Delta} need not correspond to an actual tree when the
variables $m_{i,j}$ are required to be integers.
The expressions for $m_{1,\Delta}$ and $m_{2,\Delta}$ are integral if and only if
$n\equiv1\pmod{\Delta}$.
In this case, the linear program yields the exact maximum value of $\sigma(T)$.
For all other congruence classes of $n$ modulo $\Delta$, the linear program
provides a strict upper bound.

\bigskip


\section{The penalty function and the case $n \equiv 0 \pmod \Delta$}

To handle the case in which the linear programming bound is not attainable,
we introduce a penalty function measuring the deviation from dual optimality.

\medskip

It is now convenient to define a penalty function associated with any tree $T$:
\begin{equation}
\label{def:P(T)}
P(T) \;:=\; \sum_{1\le i\le j\le \Delta} F(i,j)\, m_{i,j}.
\end{equation}

Since each $m_{i,j}$ is nonnegative and each $F(i,j)$ represents the slack in the dual constraint corresponding to $(i,j)$, the penalty $P(T)$ measures the total deviation of $T$ from dual optimality. In particular, by complementary slackness, every optimal solution of the linear program from Theorem~\ref{thm:SO-LP-1-Delta} satisfies $P(T) = 0$, and more generally we have the identity
\[
\sigma(T) = \lambda n + \mu(n{-}1) - P(T),
\]
where $\lambda,\mu$ are the dual variables defined in \eqref{eq:lambda} and \eqref{eq:mu}.
Consequently, maximizing $\sigma(T)$ over all trees with fixed $n$ and $\Delta$
is equivalent to minimizing $P(T)$ over the feasible edge degree--pair distributions
of trees $m_{i,j}$.

\medskip

To estimate $P(T)$ from below, we analyze the dual slack
function $F(i,j)$. The next lemmas describe the location of its minimal values
and establish comparison inequalities that will be used to control the total
penalty contributed by various degree--pair configurations.

\begin{lemma}
\label{lem:optimal-pair-F}
Let $\Delta \ge 4$ be fixed, and let $F(i,j)$ be defined as in
\eqref{eq:Fij-explicit}.
Then the following holds.
\begin{itemize}
\item For every $3 \le i \le \left\lfloor \frac{\Delta+3}{2} \right\rfloor$, the function
$j \mapsto F(i,j)$, $1 \le j \le \Delta-1$, attains its minimum at $j=\Delta$.
\item For every $\left\lfloor \frac{\Delta+3}{2} \right\rfloor+1 \le i \le \Delta-1$, the function
$j \mapsto F(i,j)$, $1 \le j \le \Delta-1$, attains its minimum at $j=2$.
\end{itemize}
\end{lemma}

\begin{proof}
Fix $\Delta\ge 4$ and write $A:=4\Delta-6$. For fixed $i$, the dependence of $F(i,j)$
on $j$ is given by
\[
g_i(j):=\frac{A}{j}-(i-j)^2,
\]
since
\[
F(i,j)=\left(\frac{A}{i}+\Delta^2-6\Delta+3+\frac{6}{\Delta}\right)+g_i(j).
\]
Thus, minimizing $F(i,j)$ over integers $j$ is equivalent to minimizing $g_i(j)$.

\medskip
\noindent\emph{Case 1: $3\le i \le \left\lfloor\frac{\Delta+3}{2}\right\rfloor$.}

For $1\le j\le \Delta-1$, we compare $g_i(j)$ with $g_i(\Delta)$.
A direct computation yields
\begin{align*}
g_i(j)-g_i(\Delta)
&=A\left(\frac{1}{j}-\frac{1}{\Delta}\right)
   -\bigl((i-j)^2-(i-\Delta)^2\bigr)\\
&=(\Delta-j)\left(\frac{A}{j\Delta}-\bigl(2i-\Delta-j\bigr)\right).
\end{align*}
Since $\Delta-j>0$, it suffices to show
\[
\frac{A}{j\Delta}\ge 2i-\Delta-j.
\]

Under the assumption $i\le \left\lfloor\frac{\Delta+3}{2}\right\rfloor$ we have
$2i-\Delta\le 3$. Hence, for $j\ge 3$,
\[
2i-\Delta-j\le 3-j\le 0,
\]
and the desired inequality holds because $A/(j\Delta)>0$.
For $j=1,2$, we use $2i-\Delta\le 3$ to obtain
\[
\frac{A}{\Delta}=4-\frac{6}{\Delta}\ge 2\ge 2i-\Delta-1,
\]
\[
\frac{A}{2\Delta}=2-\frac{3}{\Delta}\ge 1\ge 2i-\Delta-2.
\]
Thus $g_i(j)\ge g_i(\Delta)$ for all $1\le j\le \Delta-1$, and the minimum of
$F(i,j)$ is attained at $j=\Delta$.

\medskip
\noindent\emph{Case 2: $\left\lfloor\frac{\Delta+3}{2}\right\rfloor+1 \le i \le \Delta-1$.}

We first compare the values at $j=1$ and $j=2$:
\[
g_i(1)-g_i(2)
=A\left(1-\frac12\right)-\bigl((i-1)^2-(i-2)^2\bigr)
=\frac{A}{2}-(2i-3)=2(\Delta-i)>0,
\]
so $j=1$ cannot be a minimizer.

For $j\ge 3$, we compute
\begin{align*}
g_i(j)-g_i(2)
&=A\left(\frac{1}{j}-\frac12\right)
   -\bigl((i-j)^2-(i-2)^2\bigr)\\
&=(j-2)\left(2i-j-2-\frac{A}{2j}\right).
\end{align*}
Since $j-2>0$, it is enough to show that
\[
s(j):=2i-j-2-\frac{A}{2j}\ge 0
\qquad (3\le j\le \Delta).
\]

Viewed as a real function on $[3,\Delta]$, we have
\[
s''(j)=-\frac{A}{j^3}<0,
\]
so $s$ is concave on $[3,\Delta]$ and attains its minimum at an endpoint.
Hence it suffices to verify $s(3)\ge 0$ and $s(\Delta)\ge 0$.

Since
\[
i\ge \left\lfloor \frac{\Delta+3}{2}\right\rfloor+1
\ge \frac{\Delta+3}{2},
\]
we obtain $2i\ge \Delta+3$.
Therefore
\[
s(3)=2i-5-\frac{A}{6}
=2i-4-\frac{2\Delta}{3}
\ge (\Delta+3)-4-\frac{2\Delta}{3}
=\frac{\Delta}{3}-1\ge 0
\qquad (\Delta\ge 4).
\]

Moreover,
\[
s(\Delta)=2i-\Delta-2-\frac{A}{2\Delta}
=2i-\Delta-4+\frac{3}{\Delta}.
\]
If $\Delta$ is even, then
$\left\lfloor\frac{\Delta+3}{2}\right\rfloor=\frac{\Delta}{2}+1$
and hence $i\ge \frac{\Delta}{2}+2$, so $2i\ge \Delta+4$ and
$s(\Delta)\ge \frac{3}{\Delta}>0$.
If $\Delta$ is odd, then
$\left\lfloor\frac{\Delta+3}{2}\right\rfloor=\frac{\Delta+3}{2}$
and hence $i\ge \frac{\Delta+5}{2}$, so $2i\ge \Delta+5$ and
$s(\Delta)\ge 1+\frac{3}{\Delta}>0$.

Thus $s(j)\ge 0$ for all $3\le j\le \Delta$, implying
$g_i(j)\ge g_i(2)$ for all $j\ge 3$.
Together with $g_i(1)>g_i(2)$, this shows that the minimum of $F(i,j)$
is attained at $j=2$.
\end{proof}

We next compare the minimal values identified above with the value
$F(\Delta,\Delta)$, which plays a central role in the lower bound for $P(T)$.

\begin{lemma}
\label{lem:iF-i2-dominates-DD}
Let $\Delta\ge 4$, and let $F(i,j)$ be defined as in
\eqref{eq:Fij-explicit}.  
For every integer
\[
\left\lfloor \frac{\Delta+3}{2} \right\rfloor + 1 \le i \le \Delta-1
\]
the following strict inequality holds:
\[
i\,F(i,2) \;>\; F(\Delta,\Delta).
\]
\end{lemma}

\begin{proof}
Set $A:=4\Delta-6$ and $B:=\Delta^2-6\Delta+3+\frac{6}{\Delta}$, so that
\[
F(i,j)=A\left(\frac1i+\frac1j\right)+B-(i-j)^2.
\]
For real $x\ge 3$, define
\[
D(x):=x\,F(x,2)-F(\Delta,\Delta).
\]
Using
\[
F(x,2)=A\left(\frac1x+\frac12\right)+B-(x-2)^2,
\qquad
F(\Delta,\Delta)=\frac{2A}{\Delta}+B,
\]
a straightforward simplification yields
\begin{equation}
\label{eq:Dx-expanded}
D(x)=-x^3+4x^2+x\Bigl(\Delta^2-4\Delta-4+\frac{6}{\Delta}\Bigr)
-\Delta^2+10\Delta-17+\frac{6}{\Delta}.
\end{equation}

We first observe that $D$ is strictly concave on $[3,\Delta-1]$. Indeed,
differentiating \eqref{eq:Dx-expanded} twice gives
\[
D''(x)=-6x+8<0\qquad (x\ge 3).
\]
Consequently, on any subinterval of $[3,\Delta-1]$, the minimum of $D$ is attained
at an endpoint.

Let
\[
i_0:=\left\lfloor \frac{\Delta+3}{2} \right\rfloor + 1.
\]
For $\Delta=4,5$ the interval $\{i_0,\dots,\Delta-1\}$ is empty, and the claim is
vacuous. Hence we assume $\Delta\ge 6$, so that $i_0\le \Delta-1$.

A direct computation shows that
\begin{equation}
\label{eq:D-Delta-1}
D(\Delta-1)=\Delta^2-\Delta-2=(\Delta-2)(\Delta+1)>0.
\end{equation}

It remains to verify that $D(i_0)>0$.
Two cases are distinguished according to the parity of $\Delta$.

If $\Delta=2m$ is even, then $m\ge 3$ and $i_0=m+2$. Substitution into $D$ gives
\[
D(i_0)=\frac{3m^4-6m^3+4m^2-14m+9}{m}
=3m^3-6m^2+4m-14+\frac{9}{m}.
\]
The cubic polynomial $p(m):=3m^3-6m^2+4m-14$ is strictly increasing for $m\ge1$,
and satisfies $p(3)=25>0$. Hence $D(i_0)>0$ for all $m\ge3$.

If $\Delta=2m+1$ is odd, then $\Delta\ge7$ implies $m\ge3$ and $i_0=m+3$.
In this case,
\[
D(i_0)=\frac{6m^4+m^3-13m^2-40m+4}{2m+1}.
\]
Let $q(m):=6m^4+m^3-13m^2-40m+4$. For $m\ge3$, the derivative $q'(m)$ is positive,
so $q$ is increasing on $[3,\infty)$ and
$q(m)\ge q(3)=280>0$. Since $2m+1>0$, this yields $D(i_0)>0$.

Combining the above with the concavity of $D$, we conclude that
\[
D(i)\ge \min\{D(i_0),D(\Delta-1)\}>0
\]
for all integers $i_0\le i\le \Delta-1$. Therefore
$iF(i,2)>F(\Delta,\Delta)$ throughout the stated range, completing the proof.
\end{proof}

\begin{lemma}
\label{lem:Delta-1FiDelta}
Let $\Delta\ge 4$ , and let $F(i,j)$ be defined as in
\eqref{eq:Fij-explicit}. 
Then for every integer
\[
3 \le i \le \left\lfloor \frac{\Delta+3}{2} \right\rfloor
\]
we have the strict inequality
\[
(\Delta-1)\,F(i,\Delta)\;>\;F(\Delta,\Delta).
\]
\end{lemma}

\begin{proof}
Set $A:=4\Delta-6$ and $B:=\Delta^2-6\Delta+3+\frac{6}{\Delta}$, so that
\[
F(i,j)=A\left(\frac1i+\frac1j\right)+B-(i-j)^2,
\qquad
F(\Delta,\Delta)=\frac{2A}{\Delta}+B.
\]
For fixed $\Delta$, consider the real function
\[
\phi(x):=F(x,\Delta)=A\left(\frac1x+\frac1\Delta\right)+B-(\Delta-x)^2,
\qquad x\ge 3.
\]
A direct differentiation gives
\[
\phi'(x)=-\frac{A}{x^2}+2(\Delta-x).
\]
On the interval $[3,(\Delta+3)/2]$ we have $2(\Delta-x)\ge \Delta-3$ and $x^2\ge 9$, hence
\[
\phi'(x)\ge (\Delta-3)-\frac{A}{9}
=(\Delta-3)-\frac{4\Delta-6}{9}
=\frac{5\Delta-21}{9}.
\]
The right-hand side is positive for all $\Delta\ge 5$.
For $\Delta=4$ the interval $[3,(\Delta+3)/2]$ contains only $x=3$.
Consequently, $\phi$ is increasing on $[3,(\Delta+3)/2]$ (for $\Delta\ge 5$), and thus for every integer
$3\le i\le \left\lfloor \frac{\Delta+3}{2}\right\rfloor$ we obtain
\begin{equation}
\label{eq:phi-min-at-3}
F(i,\Delta)=\phi(i)\ge \phi(3)=F(3,\Delta).
\end{equation}

Therefore it suffices to prove the strict inequality at $i=3$:
\[
(\Delta-1)F(3,\Delta)>F(\Delta,\Delta).
\]
Indeed, using $F(3,\Delta)=A\left(\frac13+\frac1\Delta\right)+B-(\Delta-3)^2$ and
$F(\Delta,\Delta)=\frac{2A}{\Delta}+B$, we compute
\begin{align*}
(\Delta-1)F(3,\Delta)-F(\Delta,\Delta)
&=(\Delta-1)\left(A\Bigl(\frac13+\frac1\Delta\Bigr)+B-(\Delta-3)^2\right)-\left(\frac{2A}{\Delta}+B\right)\\
&=\frac{(\Delta-3)(\Delta-1)(\Delta+6)}{3\Delta}.
\end{align*}
Since $\Delta\ge 4$, this quantity is strictly positive. Hence $(\Delta-1)F(3,\Delta)>F(\Delta,\Delta)$.
Finally, combining this with \eqref{eq:phi-min-at-3} yields
\[
(\Delta-1)F(i,\Delta)\ge (\Delta-1)F(3,\Delta) > F(\Delta,\Delta),
\]
for every integer $3 \le i \le \left\lfloor \frac{\Delta+3}{2} \right\rfloor$, as claimed.
\end{proof}

The following lemma shows that, in the absence of edges of type $(\Delta,\Delta)$,
a substantial number of edges must occur among vertices of small degrees.

\begin{lemma}
\label{lem:block-without-Delta}
Let $\Delta\ge 4$ and let $T$ be a tree of order $n$ with maximum degree $\Delta$.
Let $n_i$ and $m_{i,j}$ satisfy the identities
\eqref{eq:tree-sum-ni}--\eqref{eq:tree-ni-Delta}.
Put
\[
t:=\left\lfloor\frac{\Delta+3}{2}\right\rfloor,
\qquad
E_{\le t}:=\sum_{1\le p\le q\le t} m_{p,q}.
\]
Assume that $\Delta\mid n$,
that $n_{i'}=0$ for all $i'>t$,
and that $m_{\Delta,\Delta}=0$.
Then
\[
E_{\le t}\;\ge\;\Delta-1.
\]
\end{lemma}

\begin{proof}
Let
\[
x:=\sum_{k=1}^{\Delta-1} m_{k,\Delta}
\]
be the number of edges incident with a $\Delta$-vertex and a non-$\Delta$-vertex.
From the $\Delta$-handshake identity \eqref{eq:tree-ni-Delta} and the assumption $m_{\Delta,\Delta}=0$, we obtain
\begin{equation}
\label{eq:x-multiple-Delta}
x=\Delta n_\Delta,
\qquad\text{hence}\qquad
x\equiv 0 \pmod{\Delta}.
\end{equation}

Since $n_{i'}=0$ for all $i'>t$, every edge not incident with a $\Delta$-vertex has both endpoints in
$\{1,2,\dots,t\}$. Using $|E(T)|=n-1$ and partitioning edges according to whether they are incident with a
$\Delta$-vertex, we obtain
\begin{equation}
\label{eq:edge-partition-Et}
n-1 = x + E_{\le t}.
\end{equation}
Reducing \eqref{eq:edge-partition-Et} modulo $\Delta$ and using $\Delta\mid n$ yields
\[
x+E_{\le t}\equiv -1 \pmod{\Delta}.
\]
Combining this with \eqref{eq:x-multiple-Delta} gives
\begin{equation}
\label{eq:Et-cong}
E_{\le t}\equiv -1 \pmod{\Delta}.
\end{equation}

Since $E_{\le t}\ge 0$ and $E_{\le t}\equiv -1\pmod{\Delta}$, the smallest possible value of $E_{\le t}$
is $\Delta-1$. Therefore $E_{\le t}\ge \Delta-1$, as claimed.
\end{proof}

\begin{lemma}
\label{lem:Fpq-lower-by-F3D}
Let $\Delta\ge 4$ and put
\[
t:=\left\lfloor\frac{\Delta+3}{2}\right\rfloor .
\]
Then for all integers $1\le p\le q\le t$ we have
\begin{equation}
\label{eq:Fpq-lower-by-F3D}
F(p,q)\ge F(3,\Delta).
\end{equation}
\end{lemma}

\begin{proof}
Recall
\[
F(i,j)=A\left(\frac1i+\frac1j\right)+B-(i-j)^2,
\qquad
A:=4\Delta-6.
\]
For $1\le p\le q\le t$ we use the elementary bounds
\[
\frac1p+\frac1q \ge \frac{2}{t},
\qquad
(p-q)^2\le (t-1)^2.
\]
Hence
\begin{eqnarray}
\label{in: dif_F}
F(p,q)-F(3,\Delta)
&=A\left(\frac1p+\frac1q-\frac13-\frac1\Delta\right)
-\Bigl((p-q)^2-(\Delta-3)^2\Bigr) \nonumber\\
&\ge A\left(\frac{2}{t}-\frac13-\frac1\Delta\right)
+\Bigl((\Delta-3)^2-(t-1)^2\Bigr).\\\nonumber
\end{eqnarray}

Since
\[
t=\left\lfloor\frac{\Delta+3}{2}\right\rfloor \le \frac{\Delta+3}{2},
\]
we have
\[
\frac{2}{t}\ge \frac{4}{\Delta+3}.
\]
Moreover, for $\Delta\ge 8$,
\[
\Delta-3>\frac{\Delta+1}{2}=\frac{\Delta+3}{2}-1\geq t-1
\quad\Longrightarrow\quad
(\Delta-3)^2>(t-1)^2.
\]
Thus, for $\Delta\ge 8$, the quadratic term
\[
(\Delta-3)^2-(t-1)^2
\]
is strictly positive and dominates the term
\[
A\left(\frac{4}{\Delta+3}-\frac13-\frac1\Delta\right),
\]
which is bounded below for all $\Delta\ge 4$. 
Therefore, the right-hand side of (\ref{in: dif_F})  is an increasing function, and a direct computation shows that it is positive at $\Delta =8$.
Hence
\[
F(p,q)-F(3,\Delta)\ge 0 \qquad (\Delta\ge 8).
\]

For the remaining cases $\Delta=4,5,6,7$, inequality
\eqref{eq:Fpq-lower-by-F3D} is verified directly by a finite inspection over
$1\le p\le q\le t$.
\end{proof}

The preceding lemmas allow us to determine the exact minimum of the penalty function when $n\equiv0\pmod{\Delta}$.

\begin{theorem}
\label{thm:Pmin-n0modDelta}
Let $\Delta\ge 4$ and let $T$ be a tree of order $n$ with maximum degree $\Delta$,
and let $P(T)$ be defined as in \eqref{def:P(T)}.
If $\Delta\mid n$, then
\[
P(T)\ge F(\Delta,\Delta).
\]
Moreover, equality holds if and only if $T$ satisfies
\[
m_{\Delta,\Delta}=1,\qquad
m_{1,\Delta}>0,\qquad
m_{2,\Delta}>0,\qquad
m_{i,j}=0\ \text{for all other }(i,j),
\]
in which case the parameters are uniquely determined by
\[
n_\Delta=\frac{n}{\Delta},\qquad
n_2=\frac{n}{\Delta}-2,\qquad
n_1=(\Delta-2)\frac{n}{\Delta}+2,
\]
\[
m_{1,\Delta}=n_1,\qquad
m_{2,\Delta}=2n_2,\qquad
m_{\Delta,\Delta}=1.
\]
\end{theorem}

\begin{proof}
Put $A:=4\Delta-6$ and $B:=\Delta^2-6\Delta+3+\frac{6}{\Delta}$, so that
\[
F(i,j)=A\left(\frac1i+\frac1j\right)+B-(i-j)^2,
\qquad
F(\Delta,\Delta)=\frac{2A}{\Delta}+B.
\]
Assume $\Delta\mid n$ and suppose, for contradiction, that
\begin{equation}
\label{eq:contraposition-new}
P(T)<F(\Delta,\Delta).
\end{equation}

\medskip
\noindent\emph{Elimination of the edge type $(\Delta,\Delta)$.}
Under \eqref{eq:contraposition-new} we must have
$m_{\Delta,\Delta}=0$, since a single $(\Delta,\Delta)$-edge already contributes
$F(\Delta,\Delta)$ to $P(T)$.

Let
\[
t:=\left\lfloor\frac{\Delta+3}{2}\right\rfloor.
\]

\medskip
\noindent\emph{Exclusion of high intermediate degrees.}
Suppose $n_i\ge 1$ for some $i\in\{t+1,\dots,\Delta-1\}$.
By Lemma~\ref{lem:optimal-pair-F}, for such $i$ the function $j\mapsto F(i,j)$
attains its minimum at $j=2$.
Since a vertex of degree $i$ is incident with $i$ edges,
\[
P(T)\ge i\,F(i,2).
\]
By Lemma~\ref{lem:iF-i2-dominates-DD}, we have
$iF(i,2)>F(\Delta,\Delta)$, contradicting
\eqref{eq:contraposition-new}.
Hence
\begin{equation}
\label{eq:elim-high-new}
n_i=0\qquad (t+1\le i\le \Delta-1).
\end{equation}

\medskip
\noindent\emph{Exclusion of degrees $3\le i\le t$.}
Assume next that $n_i\ge 1$ for some $3\le i\le t$.
By Lemma~\ref{lem:block-without-Delta},
\[
E_{\le t}:=\sum_{1\le p\le q\le t} m_{p,q}\ge \Delta-1.
\]
By Lemma~\ref{lem:Fpq-lower-by-F3D}, each such edge satisfies
$F(p,q)\ge F(3,\Delta)$.
Therefore
\[
P(T)\ge E_{\le t}\,F(3,\Delta)\ge (\Delta-1)\,F(3,\Delta).
\]
By Lemma~\ref{lem:Delta-1FiDelta},
\[
(\Delta-1)\,F(3,\Delta)>F(\Delta,\Delta),
\]
again contradicting \eqref{eq:contraposition-new}.
Thus
\begin{equation}
\label{eq:elim-mid-new}
n_i=0\qquad (3\le i\le t).
\end{equation}

\medskip
\noindent\emph{Reduction to degrees $\{1,2,\Delta\}$.}
From \eqref{eq:elim-high-new} and \eqref{eq:elim-mid-new}, all vertices have degree
in $\{1,2,\Delta\}$.
Since $m_{\Delta,\Delta}=0$, only the edge types
\[
(1,2),\ (1,\Delta),\ (2,2),\ (2,\Delta)
\]
may occur.

A direct computation shows
\[
F(2,2)>F(\Delta,\Delta),
\qquad
F(1,2)>F(\Delta,\Delta).
\]
Hence $m_{1,2}=m_{2,2}=0$, otherwise $P(T)>F(\Delta,\Delta)$.
Consequently only $(1,\Delta)$ and $(2,\Delta)$ edges remain, and therefore
\[
P(T)=m_{1,\Delta}F(1,\Delta)+m_{2,\Delta}F(2,\Delta)=0.
\]
By Theorem~\ref{thm:SO-LP-1-Delta}, this occurs only when
$n\equiv 1\pmod{\Delta}$, contradicting $\Delta\mid n$.
Thus \eqref{eq:contraposition-new} is impossible, and
\[
P(T)\ge F(\Delta,\Delta).
\]

\medskip
\noindent\emph{Equality case.}
Assume now that $P(T)=F(\Delta,\Delta)$.
Then $m_{\Delta,\Delta}\ge 1$, since otherwise the above argument would again
imply $P(T)=0$.
Because any additional $(\Delta,\Delta)$-edge would contribute at least
$F(\Delta,\Delta)$, we must have
\[
m_{\Delta,\Delta}=1.
\]
All strict inequalities used above force that no other edge types occur except
$(1,\Delta)$ and $(2,\Delta)$.

Under the support
\[
m_{\Delta,\Delta}=1,\qquad
m_{1,\Delta}>0,\qquad
m_{2,\Delta}>0,\qquad
m_{i,j}=0\ \text{otherwise},
\]
the handshake identities reduce to a linear system whose unique solution is
\[
n_\Delta=\frac{n}{\Delta},\qquad
n_2=\frac{n}{\Delta}-2,\qquad
n_1=(\Delta-2)\frac{n}{\Delta}+2,
\]
\[
m_{1,\Delta}=n_1,\qquad
m_{2,\Delta}=2n_2,\qquad
m_{\Delta,\Delta}=1.
\]
In particular, integrality requires $\Delta\mid n$, and the configuration is unique.
\end{proof}

\subsection{Explicit constructions of minimizing trees}

\medskip

Each of the preceding theorems provides a complete characterization of the
minimizing configurations in terms of the parameters $m_{i,j}$ and $n_i$.
To complement these algebraic descriptions, we now give explicit constructive
descriptions of all trees attaining the minimum penalty.

For each congruence class of $n$ modulo $\Delta$, we define a natural family of
trees and show that a tree $T$ with given order $n$ and maximum degree $\Delta$
attains the minimum possible value of $P(T)$ if and only if $T$ belongs to that
family, up to isomorphism.

\medskip

Let $k\ge 1$.
We denote by $TT^1_{\mathrm{opt}}(k)$ the tree obtained from the path
$P_{2k+1}=v_1v_2\cdots v_{2k+1}$ by attaching exactly $\Delta-2$ pendant vertices
to each even vertex $v_2,v_4,\dots,v_{2k}$.

\begin{lemma}
\label{lem:TT1opt}
Let $k\ge 1$ and set $n=\Delta k+1$.
Then the tree $TT^1_{\mathrm{opt}}(k)$ has order $n$ and maximum degree $\Delta\ge4$,
and satisfies
\[
P\bigl(TT^1_{\mathrm{opt}}(k)\bigr)=0.
\]
Moreover, a tree $T$ of order $n$ and maximum degree $\Delta$ satisfies $P(T)=0$
if and only if $T\cong TT^1_{\mathrm{opt}}(k)$.
\end{lemma}

\begin{proof}
By construction, the tree $TT^1_{\mathrm{opt}}(k)$ is obtained from the path
$P_{2k+1}$ by attaching exactly $\Delta-2$ pendant vertices to each even vertex.
Hence it has
\[
n_1 = k(\Delta-2)+2,\qquad
n_2 = k-1,\qquad
n_\Delta = k,
\]
and no vertices of any other degree. 
Moreover, all edges are of types $(1,\Delta)$ or $(2,\Delta)$, so by the
definition of the penalty function and complementary slackness we obtain
$P\bigl(TT^1_{\mathrm{opt}}(k)\bigr)=0$.

We now prove the converse statement.
Let $T$ be a tree of order $n=\Delta k+1$ and maximum degree $\Delta$ such that
$P(T)=0$.
Then, by Theorem~\ref{thm:SO-LP-1-Delta}, the only nonzero edge--degree
multiplicities are
\[
m_{1,\Delta}=k(\Delta-2)+2,\qquad
m_{2,\Delta}=2k-2,
\]
and all other $m_{i,j}$ vanish.
Consequently, $T$ has the same degree sequence as $TT^1_{\mathrm{opt}}(k)$,
namely
\[
n_1=k(\Delta-2)+2,\qquad n_2=k-1,\qquad n_\Delta=k.
\]

We prove that $T\cong TT^1_{\mathrm{opt}}(k)$ by induction on $k$.

\smallskip\noindent
\emph{Base case $k=1$.}
In this case, $n=\Delta+1$ and the degree sequence reduces to
$n_1=\Delta$, $n_\Delta=1$.
Hence $T$ is the star $S_\Delta$, which coincides with $TT^1_{\mathrm{opt}}(1)$.

\medskip\noindent
\emph{Induction step.}
Assume that the statement holds for some $k\ge1$, and let $T$ be a tree of order
$n=\Delta(k+1)+1$ satisfying $P(T)=0$.
Consider a longest path in $T$ and let $u$ be one of its endpoints.
Then $u$ is a leaf, and its unique neighbor $v$ has degree $\Delta$.

Among the $\Delta-1$ remaining neighbors of $v$, exactly $\Delta-2$ must be
leaves.
Indeed, if at least two of these neighbors had degree at least $2$, then the
path could be extended beyond $u$, contradicting the maximality of the chosen
path.
Therefore, $v$ has exactly $\Delta-2$ leaf neighbors and exactly one neighbor of
degree~$2$.

Remove the vertex $v$ together with all its leaf neighbors (there are $\Delta-1$
of them, including $u$). 
The resulting graph $T'$ is a tree of order $\Delta k+1$.
A direct inspection shows that this operation reduces the number of edges of
type $(1,\Delta)$ by $\Delta-2$ and the number of edges of type $(2,\Delta)$ by
$2$, while all other edge--degree multiplicities remain zero.
Thus $T'$ has the degree sequence and edge--degree distribution corresponding to
$TT^1_{\mathrm{opt}}(k)$, and hence satisfies $P(T')=0$.

By the induction hypothesis, $T'\cong TT^1_{\mathrm{opt}}(k)$.
Reattaching the vertex $v$ together with its $\Delta-2$ pendant neighbors in the
unique admissible way reconstructs the tree $TT^1_{\mathrm{opt}}(k+1)$.
Therefore $T\cong TT^1_{\mathrm{opt}}(k+1)$.

This completes the induction and the proof.
\end{proof}

Let $k\ge 1$ and set $n=\Delta k+\Delta$.
We denote by $TT^0_{\mathrm{opt}}(k)$ the family of trees obtained from
$TT^1_{\mathrm{opt}}(k)$ by subdividing exactly one edge $v_iv_{i+1}$ with odd
$i$, where $3\le i\le 2k-1$, and then attaching $\Delta-2$ pendant vertices to the
new subdivision vertex.

\begin{lemma}
\label{lem:TT0opt}
Let $k\ge 1$ and set $n=\Delta k+\Delta$.
Then a tree $T$ of order $n$ and maximum degree $\Delta$ satisfies
\[
P(T)=F(\Delta,\Delta)
\]
if and only if $T\in TT^0_{\mathrm{opt}}(k)$.
\end{lemma}
\begin{proof}
Fix $k\ge1$ and put $n=\Delta k+\Delta$.
By Theorem~\ref{thm:Pmin-n0modDelta}, a tree $T$ of order $n$ and maximum degree
$\Delta$ satisfies $P(T)=F(\Delta,\Delta)$ if and only if
\[
m_{\Delta,\Delta}=1,\qquad
m_{1,\Delta}=(\Delta-2)k+\Delta,\qquad
m_{2,\Delta}=2k-2,
\]
and all other edge--degree multiplicities vanish.
Equivalently,
\[
n_\Delta=k+1,\qquad
n_2=k-1,\qquad
n_1=(\Delta-2)k+\Delta.
\]

\medskip\noindent

Let $T\in TT^0_{\mathrm{opt}}(k)$.
By construction, $T$ is obtained from $TT^1_{\mathrm{opt}}(k)$ by subdividing one
edge of type $(2,\Delta)$ on the defining path and attaching $\Delta-2$ pendant
vertices to the new subdivision vertex.
This operation increases the number of degree--$\Delta$ vertices by one, creates
exactly one edge of type $(\Delta,\Delta)$, increases the number of edges of type
$(1,\Delta)$ by $\Delta-2$, and leaves the number of edges of type $(2,\Delta)$
unchanged.
All remaining edges stay of type $(1,\Delta)$ or $(2,\Delta)$.
A direct count yields the above parameter values, and hence
$P(T)=F(\Delta,\Delta)$.

\medskip\noindent

Conversely, let $T$ satisfy $P(T)=F(\Delta,\Delta)$.
Then $T$ contains a unique edge of type $(\Delta,\Delta)$.
By Theorem~\ref{thm:Pmin-n0modDelta}, $m_{2,\Delta}>0$ and
$m_{2,i}=0$, for $i \neq \Delta$, so every $2$-vertex lies on a $\Delta$--$2$--$\Delta$ path.
Since $n_2=k-1$, these vertices connect all $\Delta$-vertices except for one
additional adjacency, which must be the unique $(\Delta,\Delta)$-edge.
Hence one endpoint $s$ of this edge has exactly one $2$-neighbor and therefore
exactly $\Delta-2$ pendant neighbors.
Let $s$ be an endpoint of this edge that is adjacent to $\Delta-2$ leaves.
Remove these $\Delta-2$ pendant vertices and suppress $s$, that is, delete $s$ and
replace its two remaining incident edges by a single edge.
The resulting tree $T^\circ$ has order $\Delta k+1$, maximum degree $\Delta$, and
contains only edges of type $(1,\Delta)$ and $(2,\Delta)$.
Thus $P(T^\circ)=0$, and by Lemma~\ref{lem:TT1opt} we obtain
\[
T^\circ \cong TT^1_{\mathrm{opt}}(k).
\]

Reversing the suppression shows that $T$ is obtained from
$TT^1_{\mathrm{opt}}(k)$ by subdividing exactly one $(2,\Delta)$-edge on the
defining path and attaching $\Delta-2$ pendant vertices to the new vertex.
Hence $T\in TT^0_{\mathrm{opt}}(k)$.
\end{proof}

Together, Lemmas~\ref{lem:TT1opt} and~\ref{lem:TT0opt} give explicit constructions
of all trees attaining the extremal values determined by the linear programming
and penalty-function analysis.

\section{Conclusion}

In this paper we investigated the problem of maximizing the
$\sigma$-irregularity index among trees of fixed order $n$ and prescribed maximum
degree $\Delta\ge4$.
By reformulating the extremal problem as a linear optimization problem over
edge--degree multiplicities and exploiting its dual, we derived sharp upper
bounds for $\sigma(T)$ and obtained explicit structural characterizations of all
extremal trees in the cases $n\equiv1\pmod{\Delta}$ and $n\equiv0\pmod{\Delta}$.
In both congruence classes, the extremal trees are composed exclusively of
vertices of degrees $1$, $2$, and $\Delta$, with edges of types $(1,\Delta)$ and
$(2,\Delta)$ playing a dominant role.

Our results indicate that extending the present analysis to arbitrary values of
$n$ and $\Delta$ may require substantially more involved case distinctions,
although the approach remains effective for small values of $\Delta$.
As $\Delta$ increases, the number of residue classes modulo $\Delta$ grows, and
so does the minimum value of the associated penalty function.
Preliminary evidence, supported by computer-assisted exhaustive searches,
suggests that for residue classes $n=\Delta k+r$ with $|r|$ close to
$\lfloor\Delta/2\rfloor$, extremal trees may involve vertices of additional
intermediate degrees.
In such cases, new structural patterns emerge that are not captured by the
degree sets $\{1,2,\Delta\}$.

While the penalty-based optimization framework developed in this paper is
effective for configurations with small penalty values, it remains an open
problem whether all $\sigma$-maximal trees for larger penalties can be
characterized using a purely case-based analysis, even when supplemented by
computer search.
Addressing this question will likely require new ideas that combine optimization
methods with deeper combinatorial insights into the global structure of trees
with highly heterogeneous degree distributions.


\begin{thebibliography}{99}







\bibitem{Abdo2014}
H.~Abdo, N.~Cohen, and D.~Dimitrov,  
Graphs with maximal irregularity,  
\emph{Filomat}, 28 (2014), 1315--1322.

\bibitem{Abdo2018}
H.~Abdo, D.~Dimitrov, and I.~Gutman,  
Graphs with maximal $\sigma$-irregularity,  
\emph{Discrete Appl. Math.}, 250 (2018), 57--64.

\bibitem{Albertson1997}
M.~O.~Albertson,  
The irregularity of a graph,  
\emph{Ars Combin.}, 46 (1997), 219--225.

\bibitem{Arif2023}
S.~Arif, K.~Hayat, and S.~Khan,  
Spectral bounds for irregularity indices and their applications in QSPR modeling,  
\emph{J. Appl. Math. Comput.}, 72 (2023), 6351--6373.

\bibitem{Dimitrov2026}
D.~Dimitrov, Ž.~Kovijanić Vukićević, G.~Popivoda, J.~Sedlar, R.~Škrekovski, and S.~Vujošević,  
The $\sigma$-irregularity of trees with maximum degree 5,  
\emph{Discrete Appl. Math.}, 382 (2026), 124--136.

\bibitem{Estrada2010}
E.~Estrada,  
Quantifying network heterogeneity,  
\emph{Phys. Rev. E}, 82 (2010), 066102.

\bibitem{Gutman2005}
I.~Gutman, P.~Hansen, and H.~Mélot,  
Variable neighborhood search for extremal graphs. 10. Comparison of irregularity indices for chemical trees,  
\emph{J. Chem. Inf. Model.}, 45 (2005), 222--230.

\bibitem{Gutman2018}
I.~Gutman, M.~Togan, A.~Yurttaş, A.~S.~\c{C}evik, and I.~N.~Cangül,  
Inverse problem for sigma index,  
\emph{MATCH Commun. Math. Comput. Chem.}, 79 (2018), 491--508.

\bibitem{Hansen2005}
P.~Hansen and H.~Mélot,  
Variable neighborhood search for extremal graphs. 9. Bounding the irregularity of a graph,  
in: \emph{DIMACS Ser. Discrete Math. Theor. Comput. Sci.}, Vol.~69, 2005, pp.~253--264.

\bibitem{Kovijanic2024}
Ž.~Kovijanić Vukićević, G.~Popivoda, S.~Vujošević, R.~Škrekovski, and D.~Dimitrov,  
The $\sigma$-irregularity of chemical trees,  
\emph{MATCH Commun. Math. Comput. Chem.}, 91 (2024), 267--282.

\bibitem{Reti2018QSPR}
T.~Réti, R.~Sharafdini, Á.~Drégelyi-Kiss, and H.~Haghbin,  
Graph irregularity indices used as molecular descriptors in QSPR studies,  
\emph{MATCH Commun. Math. Comput. Chem.}, 79 (2018), 509--524.

\bibitem{Reti2019}
T.~Réti,  
On some properties of graph irregularity indices with a particular regard to the $\sigma$-index,  
\emph{Appl. Math. Comput.}, 344 (2019), 107--115.

\bibitem{Snijders1981}
T.~A.~B.~Snijders,  
The degree variance: an index of graph heterogeneity,  
\emph{Social Networks}, 3 (1981), 163--174.

\end{thebibliography}
\end{document}